\documentclass[a4paper,12pt]{amsart}
\usepackage{graphicx}
\usepackage{amsmath,amssymb,amsthm,amsfonts}
\usepackage{bbm}
\usepackage{amssymb}
\usepackage[active]{srcltx}
\usepackage[colorlinks=true]{hyperref}

\newtheorem{thm}{Theorem}[section]
\newtheorem{cor}[thm]{Corollary}
\newtheorem{lem}[thm]{Lemma}
\newtheorem{prop}[thm]{Proposition}

\newcommand{\R}{{\mathbb{R}}}

\newcommand{\dx}{\mathrm{d}x}

\numberwithin{equation}{section}

\begin{document}

\parindent 0pc
\parskip 6pt
\overfullrule=0pt

\title[Notes on Overdetermined Singular Problems]{Notes on Overdetermined Singular Problems}

\author{Francesco Esposito*, Berardino Sciunzi* and Nicola Soave$^+$}

\date{\today}

\address{* Università della Calabria, Dipartimento di Matematica e Informatica, Ponte Pietro  Bucci 31B, 87036 Arcavacata di Rende, Cosenza, Italy.}

\email{francesco.esposito@unical.it; sciunzi@mat.unical.it}

\address{+ Università degli studi di Torino, Dipartimento di Matematica, Via Carlo Alberto 10, 10123 Torino, Italy.}

\email{nicola.soave@unito.it}



\thanks{\emph{Acknowledgements.} F. E. and B. S. are partially supported by PRIN project 2017JPCAPN \emph{Qualitative and quantitative aspects of nonlinear PDEs} and by the INdAM - GNAMPA Project 2023 \emph{Variational and non-variational problems with lack of compactness}. \\
N. S. is partially supported by the PRIN 2022 project  2022R537CS \emph{$NO^3$ - Nodal Optimization, NOnlinear elliptic equations, NOnlocal geometric problems, with a focus on regularity} and by the INdAM - GNAMPA Project 2023, cod. CUP\_E53C22001930001 \emph{Regolarit\`a e singolarit\`a in problemi con frontiere libere}}


\begin{abstract}
We obtain some rigidity results for overdetermined boundary value problems for singular solutions in bounded domains.
\end{abstract}

\maketitle

\date{\today}

\section{Introduction}

The aim of this paper is to study qualitative properties of singular solutions to the following overdetermined problem
\begin{equation} \tag{$\mathcal{O}$} \label{problem}
\begin{cases}
-\Delta u\,=f(u)& \text{in}\quad\Omega \setminus \{0\}  \\
u> 0 &  \text{in}\quad\Omega\setminus \{0\}  \\
u=0 &  \text{on}\quad\partial \Omega\\
\partial_\nu u = \alpha < 0 &  \text{on}\quad\partial \Omega, \,
\end{cases}
\end{equation}
where $\Omega$ is a bounded smooth domain of $\mathbb{R}^n$ with
$n\geq 2$ (in particular we assume that $\partial \Omega$ is of class $\mathcal{C}^2$), $\alpha$ is a negative constant, $\nu$ is the outward normal at any point of $\partial \Omega$, and $f \in \text{Lip}_{\textrm{loc}}(\R_+)$. Moreover, here throughout the paper we assume that $0 \in \Omega$ since, by translation invariance, this actually cover the case where the solution has a generic single singular point.

This kind of elliptic problems go back to the celebrated paper of J. Serrin \cite{serrin}, where the author started to study elliptic equations under overdetermined boundary conditions. It is well known that if $u$ has a removable singularity in $0$, i.e. $u$ is a solution in the whole $\Omega$, then $\Omega$ is a ball and $u$ is radial around the origin ($u(x)=(R^2-|x|^2)/n$, where $R>0$ is the radius of the ball), see \cite{serrin} for more details. Serrin's result was also obtained by Weinberger in \cite{Weimb} in a simpler way, but Weinberger's approach is less flexible and works with a smaller class of problems. 

%
The technique that we use in this paper is a refinement of the moving plane method introduced by Serrin in the context of PDEs in \cite{serrin}. This technique goes back to the seminal paper of A. D. Alexandrov \cite{A} in his study of mean curvature surfaces. In particular, inspired by previous results in \cite{EFS,Dino}, we are able to treat the case of singular solutions in punctured domains also in the case of overdetermined problems. 

The literature on the study of monotonicity, symmetry and rigidity properties of classical solutions of elliptic problems via the moving planes method increased exponentially after the celebrated paper by B. Gidas, W. M. Ni and L. Nirenberg \cite{GNN}, where the authors considered semilinear Dirichlet problems in bounded domains. 

For singular problems much less is known. The first contribution in this direction, regarding solutions with finite energy, can be probably ascribed to S. Terracini in \cite{Ter}, who considered a problem with a singular potential at a point, posed in the whole space $\R^n$. The technique was then adapted in \cite{CLN2} to deal with the Dirichlet problems with a point-wise singularity. More recently, the cases of a smooth $(n-2)$-dimensional singular set, and of a singular set with zero capacity, have been considered in \cite{Dino} and \cite{EFS}, respectively, without energy assumptions on the solutions, namely considering non removable singularities. As far as overdetermined problems with singularities is concerned, we refer the reader to \cite{Aless}, where the authors considered problems of type
\[
\begin{cases}
\mathrm{div}(a(|\nabla u|) \nabla u) = 0 & \text{in }\Omega \setminus \{0\} \\
u=0 & \text{on $\partial \Omega$} \\
\partial_\nu u = \alpha & \text{on $\partial \Omega$};
\end{cases}
\]
and to \cite{AgoMag1, AgoMag2}, where the authors considered an overdetermined problem for the Green's function in the plane. In particular, up to now, there was no result for equations with a rather general locally Lipschitz right-hand side, or with $f \equiv 1$, which is the natural counterpart of Serrin's problem. In this paper, by combining the original Serrin's argument \cite{serrin} and the recent approach developed in \cite{EFS}, we provide such results in Theorem \ref{thm:main} and Corollary \ref{cor:class} below.

\begin{thm} \label{thm:main}
Let $\Omega \subset \R^n$ be a bounded domain of class $\mathcal{C}^2$ with $n \geq 2$. Let $u \in C^2(\overline{\Omega} \setminus \{0\})$ be a solution to \eqref{problem} with a non-removable singularity at $0$\footnote{Here we mean that the solution $u$ does not admit a smooth extension all over the domain $\Omega$. Namely it is not possible to find $\tilde{u} \in H^1(\Omega)$ with $u \equiv \tilde u$ in $\Omega$.}. Then, $\Omega$ is a ball centered at the origin, and $u$ is radially symmetric and decreasing.
\end{thm}

As a consequence we deduce also the complete classification for the torsion problem:

\begin{cor}\label{cor:class}
	Let $\Omega \subset \R^n$ be a bounded domain of class $\mathcal{C}^2$ with $n \geq 2$. Let $u \in C^2(\overline{\Omega} \setminus \{0\})$ be a solution to \eqref{problem} with a non-removable singularity at $0$ and $f \equiv 1$. Then, $\Omega = B_R(0)$ for some $R>0$, and 
	\begin{equation*}
		u=u_R+u_F,
	\end{equation*}
	where $u_R(x)=(R^2-|x|^2)/n$ is the unique radial solution to
	\begin{equation*}
		\begin{cases}
			-\Delta u_R=1 & \text{in } \Omega\\
			u_R=0 & \text{on } \partial \Omega,
		\end{cases}
	\end{equation*}
	and $u_F$ is a fundamental solution given by
	\begin{equation*}
		u_F(x)=u_F(|x|) :=
		\begin{cases}
			 \displaystyle \hat C\left( \frac{1}{|x|^{n-2}} - \frac{1}{R^{n-2}} \right) & \text{if } n \geq 3,\\
			 \displaystyle -\hat C(\ln |x| - \ln R) & \text{if } n \geq 3,
		\end{cases}
	\end{equation*}
	for some positive constant $\hat{C}$ depending on the boundary data.
\end{cor}

\section{Notations and preliminary results} \label{notations}

The aim of this section, is to recall some notations typical of the moving plane technique. Since our problem \eqref{problem} is invariant up to isometries, we fix the direction $\eta=e_1$ and for any real number $\lambda$ we set
\begin{equation*}
\Omega_\lambda=\{x\in \Omega:x_1 <\lambda\}
\end{equation*}
\begin{equation*}
x_\lambda= R_\lambda(x)=(2\lambda-x_1,x_2,\ldots,x_n)
\end{equation*}
which is the reflection through the hyperplane $T_\lambda :=\{ x_1=
\lambda\}$. Moreover, we set 
\begin{eqnarray*}
&&\ \ a =\inf _{x\in\Omega}x_1\\
&& \Omega'_\lambda = R_\lambda (\Omega_\lambda)\\
&& \ 0_\lambda  =(2 \lambda, 0, \dots, 0)\\
&& \ \gamma_\lambda =\partial \Omega_\lambda'\cap \{x_1 > \lambda\}.
\end{eqnarray*}
Now, we observe as in \cite{serrin} that the following fact holds: 

\textit{Since $\Omega$ is of class $\mathcal{C}^2$, if $\lambda \in (a, a + \delta)$ (for $\delta> 0$ sufficiently small), then $\Omega_\lambda' \subset \Omega$. Moreover, the outward normal $\nu_\lambda'$ on $\gamma_\lambda$ is such that $\langle \nu_\lambda', e_1 \rangle > 0$, and $0 \not \in \overline{\Omega}_\lambda$.} 

These conditions remain true as $\lambda<0$ increases, until $\lambda$ reaches a value $\bar \lambda$ for which
one of the following situations occur:
\begin{itemize}
	\item[\textbf{(i)}] $\overline{\Omega}_{\bar \lambda}'$ becomes internally tangent to $\partial \Omega$ at some point $P \not \in T_{\bar \lambda}$, and $\bar \lambda < 0$;
	
	\item[\textbf{(ii)}] it happens that $\langle \nu_{\bar \lambda}', e_1 \rangle = 0$ at some point $Q$ of $\overline{\gamma_{\bar \lambda}} \cap T_{\bar \lambda}$ with $\bar \lambda < 0$. This fact means that $T_{\bar \lambda}$ is orthogonal to $\partial \Omega$ at $Q$;
	
	\item[\textbf{(iii)}] $\bar \lambda=0$.
\end{itemize}


Recalling that our problems is invariant up to rotations and reflections, we will exploit the fact that $u_\lambda(x) := u(x_\lambda)$ is a solution to
\begin{equation} \tag{$\mathcal{O}_\lambda$} \label{problemRefl}
	\begin{cases}
		-\Delta u_\lambda\,=f(u_\lambda)& \text{in}\quad R_\lambda(\Omega) \setminus \{0_\lambda\}  \\
		u_\lambda> 0 &  \text{in}\quad R_\lambda(\Omega)\setminus \{0_\lambda\}  \\
		u_\lambda=0 &  \text{on}\quad\partial R_\lambda(\Omega)\\
		\partial_{\nu_\lambda} u_\lambda = \alpha < 0 &  \text{on}\quad\partial R_\lambda(\Omega),
	\end{cases}
\end{equation}
where $\nu_\lambda$ is the outward normal at each point of $\partial R_\lambda (\Omega)$.
Notice in particular that 
\begin{equation}\label{eq:weaksol}
	\int_{\Omega} \langle \nabla u, \nabla \varphi \rangle \, \dx = \int_{\Omega} f(u) \cdot \varphi \, \dx, \qquad \forall \varphi \in C^1_c(\Omega \setminus \{0\}),
\end{equation}
and
\begin{equation}\label{eq:weaksolrefl}
	\int_{R_\lambda(\Omega)} \langle \nabla u_\lambda, \nabla \varphi \rangle \, \dx = \int_{R_\lambda(\Omega)} f(u_\lambda) \cdot \varphi \, \dx, \qquad \forall \varphi \in C^1_c(R_\lambda(\Omega) \setminus \{0_\lambda\}).
\end{equation}
Finally, for any  $a<\lambda<0$, we define
\begin{equation*}
w_\lambda\,:=\,u_\lambda-u,
\end{equation*}
which is well defined in $\Omega \cap R_\lambda(\Omega)$. In particular, for every $a<\lambda<\bar \lambda$, it is well defined in $\Omega_\lambda$ (the fact that $\Omega_\lambda' \subset \Omega$ is equivalent to the fact that $\Omega_\lambda \subset R_\lambda(\Omega)$)

Now, it is well known that
$\underset{\R^n}{\operatorname{Cap}_2}(\{0\})=0$, and $\underset{\R^n}{\operatorname{Cap}_2}(\{0_{\lambda}\})=0$. Using the approach in \cite{EFS}, thanks to the definition of $2$-capacity, for any open neighborhood $\mathcal{B}^{\lambda}_{\epsilon}$ of
$\{0,0_\lambda\}$, chosen in such a way that $\{\mathcal{B}^{\lambda}_{\epsilon}\}$ is increasing in $\varepsilon>0$, and the diameter of any connected component of $\mathcal{B}^{\lambda}_{\epsilon}$ is proportional to $\varepsilon$, we can construct a cutoff function $\psi_{\lambda,\varepsilon} \in C^{0,1}(\R^n, [0,1])$ such that
\begin{equation}\label{eq:cutoff}
	\begin{cases}
		\psi_{\lambda,\varepsilon} = 0 \ \ \ \  \text{in } \mathcal{B}_\delta^\lambda,\\
		\psi_{\lambda,\varepsilon} = 1 \ \ \ \  \text{in } (\Omega \cup \Omega'_\lambda) \setminus \mathcal{B}_\varepsilon^\lambda,\\
		\displaystyle \int_{\mathcal{B}^{\lambda}_{\epsilon}} |\nabla
		\psi_{\lambda,\varepsilon} |^2 dx < \varepsilon,
	\end{cases}
\end{equation}
for some neighborhood $\mathcal{B}^{\lambda}_{\delta} \subset \mathcal{B}^{\lambda}_{\varepsilon}$ of $\{0,0_\lambda\}$. Having in mind this construction, we recall two technical lemmas whose proof is contained in \cite{EFS}.

\begin{lem}[Lemma 3.1, \cite{EFS}]\label{lem:aiuto1} Let $\lambda \in (a,\bar \lambda)$ such that $0_\lambda \in \overline{\Omega}_{\lambda}$, and consider the function
	$$\varphi_{\lambda, \varepsilon}\,:= \begin{cases}
		\, w_\lambda^- \psi_{\lambda,\varepsilon}^2  & \text{in}\quad\Omega_\lambda, \\
		0 &  \text{in}\quad \R^n \setminus \Omega_\lambda,
	\end{cases}$$
	where $\psi_{\lambda,\varepsilon}$ is the cutoff function introduced in \eqref{eq:cutoff}. Then, $\varphi_{\lambda, \varepsilon} \in H^1_0(\Omega_\lambda)$, and a.e. on $ \Omega \cap R_{\lambda}(\Omega)$,
	\begin{equation}\label{gradvarphi2}
		\nabla \varphi_{\lambda, \varepsilon} = \psi_{\lambda,\varepsilon}^2  (\nabla w_\lambda \mathbbm{1}_{supp(w_\lambda^-) \cap supp(\varphi_{\lambda,\varepsilon})}) +
		2 (w_\lambda^- \mathbbm{1}_{supp(\varphi_{\lambda,\varepsilon})}) \psi_{\lambda,\varepsilon} \nabla  \psi_{\lambda,\varepsilon}.
	\end{equation}
\end{lem}

\begin{lem}[Lemma 3.2, \cite{EFS}]\label{lem:aiuto2}
	Under the assumptions of Theorem \ref{thm:main}, let $\lambda \in (a,\bar \lambda)$. Then $ w_\lambda^-\in H^1_0(\Omega_\lambda)$ and
	\begin{equation*}
		\int_{\Omega_\lambda}|\nabla w_\lambda^-|^2\,dx\leq c(f,{\vert \Omega \vert}, \|u\|_{L^\infty(\Omega_\lambda)}),
	\end{equation*}
	{where $\vert \Omega \vert$ denotes the $n$-dimensional Lebesgue measure of $ \Omega$. }
\end{lem}

Finally, we end this section recalling the celebrated corner's lemma of J. Serrin stated and proved in his famous paper \cite{serrin}.

\begin{lem}[Lemma 1, \cite{serrin}]\label{lem:corner}
	Let $D^*$ be a domain with $\mathcal{C}^2$ boundary and let $T$ be a hyperplane containing the normal to $\partial D^*$ at some point $Q$. Let $D$ then denote the portion of $D^*$ lying on some particular side of $T$.
	
	Suppose that $w$ is of class $C^2$ in the closure of $D$ and satisfies 
	\begin{equation*}
		\begin{cases}
			\Delta w \leq 0 &\text{ in } D\\
			w \ge 0 & \text{ in } D\\
			w(Q)=0.
		\end{cases}
	\end{equation*}
Let $\nu$ be any direction at $Q$ which enters $D$ non-tangentially. Then
	 $$\text{either }\quad \frac{\partial w}{\partial \mu}(Q) > 0, \quad  \text{ or } \quad \frac{\partial^2 w}{\partial \mu^2} (Q) > 0,$$
	unless $w\equiv 0$.
\end{lem}

\section{Proof of symmetry and  classification results}\label{mainproofsec}

The aim of this section is to prove Theorem \ref{thm:main} via a fine version of the moving plane method.

\begin{proof}[Proof of Theorem \ref{thm:main}]

First of all we recall that, since $\Omega$ is of class $\mathcal{C}^2$, then for $\lambda>a$ sufficiently close to $a$, we have that $\Omega_\lambda' \subset \Omega$ and that the outward normal $\nu_\lambda'$ to $\gamma_\lambda$ is such that $\langle \nu_\lambda', e_1 \rangle > 0$. Moreover, in this setting, $0 \not \in \Omega_\lambda$. Hence, we can define the set
\begin{equation}\label{def Lambda}
\Lambda:=\{a<\lambda< \bar \lambda : w_\tau> 0\,\,\,\text{in}\,\,\,\Omega_\tau\setminus
\{0_\tau\}\,\,\,\text{for all $\tau\in(a,\lambda]$}\}
\end{equation}
and to start with the moving plane procedure, we have to prove
that:

\textbf{Step 1:  $\Lambda \neq \emptyset$}. Fix a $ \bar \sigma > 0$ sufficiently small, such that
$0_{\lambda} \not \in \Omega$ for any $a< \lambda < a + \sigma$, with $\sigma \in (0, \bar \sigma)$. For any $ \lambda$ in this range we have that both $u$ and $u_\lambda$ are regular in $\Omega_\lambda$, and
\begin{equation}\nonumber
\begin{cases}
-\Delta w_\lambda + c_\lambda(x) w_\lambda = 0 & \text{in } \Omega_\lambda 	\\
w_\lambda \geq 0 & \text{on } \partial \Omega_\lambda,
\end{cases}
\end{equation}
where 
\begin{equation}\nonumber
	c_\lambda(x):=\begin{cases} \displaystyle 
		-\frac{f(u_\lambda(x))-f(u(x))}{u_\lambda(x)-u(x)} & \text{if } u(x) \neq u_\lambda(x) 	\\
		0 &  \text{if } u(x) = u_\lambda(x),
	\end{cases}
\end{equation}
and $c_\lambda \in L^\infty(\Omega_\lambda)$ as $f \in \textrm{Lip}_{\textrm{loc}}(\R)$. Actually, we can say even more: there exists $C(\bar \sigma)>0$ such that
$$\|c_\lambda\|_{L^\infty(\Omega_\lambda)} \leq C, \ \text{for any } \lambda \in (a, a+\bar \sigma).$$
Hence, up to fix $\bar \sigma$ smaller, we are able to apply the weak comparison principle in small domains in order to deduce that
$$w_\lambda \geq 0 \text{ in } \Omega_\lambda$$
for any $\lambda \in (a, a+\bar \sigma)$. We point out that the same monotonicity result can be achieved applying the Hopf boundary lemma to $w_\lambda$ in a tubular neighborhood $\mathcal{I}$ of $\partial \Omega$ intersected with $\Omega_\lambda$. Moreover, since $w_\lambda>0$ in $\partial \Omega_\lambda \cap \{x_1<\lambda\}$, by the strong maximum principle we deduce that
$$w_\lambda > 0 \text{ in } \Omega_\lambda$$
for any $\lambda \in (a, a+\bar \sigma)$. The last one immediately implies that $\Lambda \neq \emptyset$.
%

%
%
%

%
Since it is not empty, by its own definition \eqref{def Lambda} it follows that $\Lambda$ is an interval, and we can define
$$\lambda^*:= \sup \Lambda \in (a, \bar \lambda].$$

\textbf{Step 2:} we claim that $\lambda^* =\bar \lambda$. 

Let us assume by contradiction that $\lambda^* < \bar \lambda$. This fact obviously implies that $\lambda^*<0$. We claim that there exists $\bar \sigma > 0$ such that $u \leq u_{\lambda^*+\sigma}$ in $\Omega_{\lambda^* + \sigma} \setminus \{0_{\lambda^*+\nu}\}$ for any $0<\sigma<\bar \sigma$. By continuity we know that $w_{\lambda^*} \geq 0$ in $\Omega_{\lambda^*} \setminus \{0_{\lambda^*}\}$; by the strong maximum principle we immediately get that $w_{\lambda^*}>0$ in $\Omega_{\lambda^*} \setminus \{0_{\lambda^*}\}$. Let us  consider any compact set $\mathcal{K} \subset 
\Omega_{\lambda^*} \setminus \{0_{\lambda^*}\}$, to be properly chosen later; thanks to the uniform continuity, we can find $\bar \sigma= \bar \sigma(\mathcal{K}, \lambda^*)>0$ sufficiently small such that $\mathcal{K} \subset \Omega_{\lambda^*+ \sigma} \setminus \{0_{\lambda^*+ \sigma}\}$, and $w_{\lambda^*+\sigma}>0$ in $\mathcal{K}$ for any $0<\sigma<\bar \sigma$. 

Let us consider the cutoff function given by Lemma \ref{lem:aiuto1},
$$\varphi_{\lambda^*+\sigma, \varepsilon}\,:= \begin{cases}
	\, w_{\lambda^*+\sigma}^- \psi_{\lambda^*+\sigma,\varepsilon}^2  & \text{in}\quad\Omega_{\lambda^*+\sigma}, \\
	0 &  \text{in}\quad \R^n \setminus \Omega_{\lambda^*+\sigma}.
\end{cases}$$
We can plug $\varphi_{\lambda^*+\sigma, \varepsilon}$ as test function in \eqref{eq:weaksol} and \eqref{eq:weaksolrefl} so that, by subtracting, we get
\begin{equation*}\label{eq:diff1}
	\int_{\Omega_{\lambda^*+\sigma}} \langle \nabla w_{\lambda^*+\sigma}, \nabla \varphi_{\lambda^*+\sigma, \varepsilon} \rangle \, \dx = \int_{\Omega_{\lambda^*+\sigma}} [f(u_{\lambda^*+\sigma})-f(u)] \cdot \varphi_{\lambda^*+\sigma, \varepsilon} \, \dx.
\end{equation*}
By \eqref{gradvarphi2}, we obtain
\begin{equation*}\label{eq:diff2}
	\begin{split}
\int_{\Omega_{\lambda^*+\sigma}} |\nabla w_{\lambda^*+\sigma}^-|^2 \psi_{\lambda^*+\sigma,\varepsilon}^2 \, \dx &\leq  2 \int_{\Omega_{\lambda^*+\sigma}}  w_{\lambda^*+\sigma}^- \psi_{\lambda^*+\sigma,\varepsilon} |\nabla w_{\lambda^*+\sigma}^-| |\nabla \psi_{\lambda^*+\sigma,\varepsilon}|  \, \dx \\
& \qquad + \int_{\Omega_{\lambda^*+\sigma}} |f(u_{\lambda^*+\sigma})-f(u)| w_{\lambda^*+\sigma}^-\psi_{\lambda^*+\sigma,\varepsilon}^2  \, \dx.
\end{split}
\end{equation*}
By using on the right hand side the weighted Young's inequality, we deduce that
\begin{multline*}
	\int_{\Omega_{\lambda^*+\sigma}} |\nabla w_{\lambda^*+\sigma}^-|^2 \psi_{\lambda^*+\sigma,\varepsilon}^2 \, \dx \\
	 \leq 8 \int_{\Omega_{\lambda^*+\sigma}} |\nabla \psi_{\lambda^*+\sigma,\varepsilon}|^2  (w_{\lambda^*+\sigma}^-)^2\, \dx + 2\mathcal{L}_f \int_{\Omega_{\lambda^*+\sigma}} (w_{\lambda^*+\sigma}^-)^2 \psi_{\lambda^*+\sigma,\varepsilon}^2  \, \dx,
\end{multline*}
where $\mathcal{L}_f$ is the Lipschitz constant of $f$. 

Now, we recall that $w_{\lambda^*+\sigma}>0$ in $\mathcal{K}$ (which means $w_{\lambda^*+\sigma}^-=0$ in $\mathcal{K}$), hence the last inequality becomes
\begin{multline*}
		\int_{\Omega_{\lambda^*+\sigma} \setminus \mathcal{K}} |\nabla w_{\lambda^*+\sigma}^-|^2 \psi_{\lambda^*+\sigma,\varepsilon}^2 \, \dx \\ \leq 8 \int_{\Omega_{\lambda^*+\sigma} \setminus \mathcal{K}} |\nabla \psi_{\lambda^*+\sigma,\varepsilon}|^2  (w_{\lambda^*+\sigma}^-)^2\, \dx + 2\mathcal{L}_f \int_{\Omega_{\lambda^*+\sigma} \setminus \mathcal{K}} (w_{\lambda^*+\sigma}^-)^2 \psi_{\lambda^*+\sigma,\varepsilon}^2  \, \dx.
\end{multline*}
By definition of $\psi_\varepsilon$, see \eqref{eq:cutoff}, and the fact that $0 \leq u_{\lambda^*+\nu} \leq u$ in $\Omega_{\lambda^*+\sigma} \cap \text{supp}(w_{\lambda^*+\sigma}^-)$, this gives
\begin{equation*}\label{eq:diff5}
	\int_{\Omega_{\lambda^*+\sigma} \setminus \mathcal{K}} |\nabla w_{\lambda^*+\sigma}^-|^2 \psi_{\lambda^*+\sigma,\varepsilon}^2 \, \dx \leq  8 \|u\|_{L^\infty(\Omega_{\lambda^*+\sigma})}^2  \varepsilon + 2\mathcal{L}_f \int_{\Omega_{\lambda^*+\sigma} \setminus \mathcal{K}} (w_{\lambda^*+\sigma}^-)^2 \psi_{\lambda^*+\sigma,\varepsilon}^2  \, \dx\,;
\end{equation*}
Hence, thanks to Fatou's lemma, we can pass to the limit as $\varepsilon \rightarrow 0^+$ in order to deduce that
\begin{equation*}\label{eq:diff6}
	\begin{split}
		\int_{\Omega_{\lambda^*+\sigma} \setminus \mathcal{K}} |\nabla w_{\lambda^*+\sigma}^-|^2  \, \dx \leq 2\mathcal{L}_f \int_{\Omega_{\lambda^*+\sigma} \setminus \mathcal{K}} (w_{\lambda^*+\sigma}^-)^2  \, \dx.
	\end{split}
\end{equation*}
Finally, the Poincar\'e's inequality on the right hand side gives
\begin{equation*}\label{eq:diff7}
	\begin{split}
		\int_{\Omega_{\lambda^*+\sigma} \setminus \mathcal{K}} |\nabla w_{\lambda^*+\sigma}^-|^2  \, \dx \leq 2\mathcal{L}_f C_p^2(|\Omega_{\lambda^*+\sigma} \setminus \mathcal{K}|) \int_{\Omega_{\lambda^*+\sigma} \setminus \mathcal{K}} |\nabla w_{\lambda^*+\sigma}^-|^2  \, \dx,
	\end{split}
\end{equation*}
where $C_p(\cdot)$ is the Poincar\'e's constant. Now, we can fix $\mathcal K \subset 
\Omega_{\lambda^*} \setminus \{0_{\lambda^*}\}$ in such a way that it occurs
$$2\mathcal{L}_f C_p^2(|\Omega_{\lambda^*+\sigma} \setminus \mathcal{K}|)<1$$
for every $\sigma \in (0, \bar \sigma)$ sufficiently small (it is sufficient to take $|\Omega_{\lambda^*} \setminus \mathcal{K}|$ sufficiently small). Hence, for any such $\sigma$ we get
$$\int_{\Omega_{\lambda^*+\sigma} \setminus \mathcal{K}} |\nabla w_{\lambda^*+\sigma}^-|^2  \, \dx \leq 0,$$
namely $w_{\lambda^*+\sigma} \geq 0$, in contradiction with the definition of $\lambda^*$. Hence, it is necessary that $\lambda^*=\bar \lambda$.

\textbf{Step 3:} we prove that  $\bar \lambda=0$. 

In Step 2 we proved that 
$$w_\lambda > 0 \; \text{in } \Omega_\lambda \setminus \{0_\lambda\} \quad \text{for any $\lambda \in (a, \bar \lambda)$}.$$
Arguing by contradiction, assume that $\bar \lambda < 0$. Notice at fist that $w_{\bar \lambda} \ge 0$ in $\Omega_{\bar \lambda} \setminus \{0_\lambda\}$. If $w_{\bar \lambda} \equiv 0$, then, taking into account the fact that $u>0$ in $\Omega$, we have that $\Omega$ is symmetric about $T_{\bar \lambda}$. However, in this case $0_{\bar \lambda} \in \Omega_{\bar \lambda}$, and this is not possible, since $w_{\bar \lambda} \equiv 0$ implies that $u$ must be singular on both $0$ and $0_{\bar \lambda}$ (note that $0 \neq 0_{\bar \lambda}$, since $\bar \lambda<0$). Therefore, assuming $\bar \lambda <0$ we have that $w_{\bar \lambda} \not \equiv 0$ in $\Omega_{\bar \lambda} \setminus \{0_{\bar \lambda}\}$, and either $\textbf{(i)}$ or $\textbf{(ii)}$ occurs in the definition of $\bar \lambda$.

First of all we focus our attention on $\textbf{(i)}$. In this case, if $0_\lambda \not \in \overline{\Omega_{\bar \lambda}}$, then one can conclude as in the paper by Serrin \cite{serrin}; namely, since $w_{\bar \lambda} \not \equiv 0$, thanks to the strong maximum principle we deduce that $w_{\bar \lambda} >0$ in $\Omega_{\bar \lambda}$ (note that locally symmetric regions cannot occur since $u>0$ in $\Omega$, $u=0$ on $\partial \Omega$, and $\Omega$ is connected). Thus, on one side by the Hopf boundary lemma we have that the outer normal $\partial_\nu w_{\bar \lambda}(P_{\bar \lambda})<0$; but on the other side, the overdetermined Neumann condition gives $\partial_\nu w_{\bar \lambda}(P_{\bar \lambda})=0$, a contradiction. 

If $0_{\bar \lambda} \in \overline{\Omega_{\bar \lambda}}$, we are still able to deduce that $w_{\bar \lambda} > 0$ in $\Omega_{\bar \lambda}  \setminus \{0_{\bar \lambda}\}$. To prove this fact, we recall that $w_{\bar \lambda}$ is nonnegative and nontrivial in the open connected set $\Omega_{\bar \lambda}  \setminus \{0_\lambda\}$; thanks to this and to the strong maximum principle, we deduce again that $w_{\bar \lambda} > 0$ in $\Omega_{\bar \lambda}  \setminus \{0_\lambda\}$. Since moreover $0_{\bar \lambda} \neq P_{\bar \lambda}$ (otherwise $P =0$, with $P \in \partial \Omega$, which is not the case), we obtain the same contradiction as before, by using the Hopf boundary lemma in $P_\lambda$. Therefore, case $\textbf{(i)}$ in the definition of $\bar \lambda$ cannot take place.

Now, let us analyze $\textbf{(ii)}$. At the point $Q$ we cannot apply the Hopf boundary lemma on $w_{\bar \lambda}$, since in that point the direction $e_1$ is orthogonal to $\partial \Omega$. Moreover, we note that 
$$w_{\bar \lambda} (Q) = u_{\bar \lambda}(Q)-u(Q)=0,$$
and that there exist a neighborhood of $Q$ which does not contain $0_{\bar \lambda}$ or $0$. Thus, in this case the contradiction follows exactly as in Serrin's paper \cite{serrin} (pages 307-308). Essentially, on one side it can be proved that $w_{\bar \lambda}$ has a zero of second order at $Q$. Then, thanks to Lemma \ref{lem:corner} we reach a contradiction since at $Q$ at least one first or second directional derivative must be negative. Hence also case \textbf{(ii)} in the definition of $\bar \lambda$ cannot take place.

In conclusion, we reached a contradiction with the fact that $\bar \lambda<0$. 
That is, the only possibility is that we are in case $\textbf{(iii)}$, of the definition of $\bar \lambda$: $\bar \lambda=0$, $0_{\bar \lambda}=0$, and $w_{\bar \lambda} \ge 0$ in $\Omega_{\bar \lambda}$.

\textbf{Step 4: conclusion}. Since the moving plane procedure can be
performed in the same way but in the opposite direction, then $u(x_1,x')=u(-x_1,x')$, and the domain $\Omega$ is symmetric about $T_0$. Moreover, the solution is
increasing in the $x_1$-direction in $\{x_1<0\}$, and by the strong maximum principle for the linearized equation we also have that $\partial_{x_1} u > 0$ in $\{x_1<0\}$. Now, we can repeat verbatim this argument replacing $e_1$ with any direction $\nu \in \mathbb{S}^{n-1}$, and hence we obtain the thesis of Theorem \ref{thm:main}.
\end{proof}

Before proving the classification result, we need the following comparison argument.

\begin{prop}\label{lem:comparison} 
	Let $u \in C^2(\overline{B_R(0)} \setminus \{0\})$ be a solution to \eqref{problem} with $f \equiv 1$.
	Then 
	$$u > u_R \text{ in } B_R(0) \setminus \{0\},$$
	where $u_R$ is the same of Corollary \ref{cor:class}.
\end{prop}

\begin{proof}
	To prove this result it is sufficient to take $\varphi=(u_R-u)^+ \cdot \psi_{0,\varepsilon}^2$ as test function in the weak formulation of $u$ and $u_R$, where $\psi_{0,\varepsilon}$ is defined in \eqref{eq:cutoff}. By subtracting both the weak formulations, arguing similarly as in the Step 2 of Theorem \ref{thm:main}, and using the fact that $\Delta U = \Delta u_R$ in $\Omega \setminus \{0\}$, we get that $u \geq u_R \text{ in } B_R(0) \setminus \{0\}$. Finally, using the strong maximum principle we get the thesis.
\end{proof}

Now, we are ready to prove Corollary \ref{cor:class}.

\begin{proof}[Proof of Corollary \ref{cor:class}]
	By Theorem \ref{thm:main} we have that if $u$ is a solution to \eqref{problem} in some smooth domain $\Omega$, then $u(x)=u(|x|)$ and $\Omega = B_R(0)$.
	Let 
	$$u_C = u_R+u_F^C,$$
	where we recall that
		where $u_R=(R^2-|x|^2)/n$ is the unique radial solution to
	\begin{equation*}
		u_R(x)=\frac{R^2-|x|^2}{n} \quad \text{and} \quad 		u_F^C(x)=
		\begin{cases}
			\displaystyle  C\left( \frac{1}{|x|^{n-2}} - \frac{1}{R^{n-2}} \right) & \text{if } n \geq 3,\\
			\displaystyle -C(\ln |x| - \ln R) & \text{if } n \geq 3,
		\end{cases}
	\end{equation*}
	for any $C \in \R$. We observe that $u_C$ solves the following problem 
	\begin{equation*}
		\begin{cases}
			-\Delta u_C=1 & \text{in } \Omega \setminus \{0\}\\
			 u_C=0 & \text{on } \partial \Omega.
		\end{cases}
	\end{equation*}
	Thanks to Proposition \ref{lem:comparison}, we observe that any solution $u$ to \eqref{problem} with $f\equiv 1$ satisfies
	$$u > u_R \text{ in } B_R(0) \setminus \{0\}.$$
	Consequently,
	$$\partial_{-\eta} u > \partial_{-\eta} u_R \text{ on } \partial B_R(0),$$
	where $-\eta$ is the inward pointing normal to $\partial B_R(0)$. Therefore there exists $\hat C \in \R$, such that
	$$\partial_\eta u = \partial_\eta u_{\hat C} \quad \text{on } \partial \Omega.$$
	Now, we deduce that 
	$$u \equiv u_{\hat C}.$$
	This fact immediately follows by Cauchy's Theorem for ODEs. In fact, we can see that both $u$ and $u_{\hat C}$ solve the following Cauchy problem
	\begin{equation*}
		\begin{cases}
			-(u'r^{n-1})'=r^{n-1} & \text{in } (0, R]\\
			u(R)=0 \\
			u'(R)=-u_{\eta}(R) = -\alpha.
		\end{cases}
	\end{equation*}
	Hence, $u \equiv u_{\hat C}$ in $(0,R]$.
\end{proof}

\bigskip

\end{document}